\documentclass[11pt]{amsart}
\usepackage{amscd}
\usepackage{amsmath, amssymb}
\usepackage{verbatim}
\usepackage{amsfonts}

\numberwithin{equation}{section}

\newtheorem{theorem}{Theorem}[section]

\newtheorem{lemma}{Lemma}[section]

\newcommand{\ve}{\varepsilon}

\theoremstyle{definition}

\newcommand{\Sym}{\textrm{Sym}}

\theoremstyle{remark}

\begin{document}

\title[Constant Rank Theorem]{On a constant rank theorem \\ for nonlinear elliptic PDEs}

\author[G. Sz\'ekelyhidi]{G\'abor Sz\'ekelyhidi}
\address{Department of Mathematics, University of Notre Dame, 255 Hurley, Notre Dame, IN 46556}
\author[B. Weinkove]{Ben Weinkove}
\address{Department of Mathematics, Northwestern University, 2033 Sheridan Road, Evanston, IL 60208}
\thanks{Supported in part by National Science Foundation grants DMS-1306298 and DMS-1350696 (G.Sz.) and DMS-1406164 (B.W.). Part of this work was carried out while the first-named author was visiting the Mathematics Department of Northwestern University and he thanks them for their hospitality.}

\begin{abstract}
We give a new proof of Bian-Guan's constant rank theorem  for nonlinear elliptic equations.  Our approach is to use  
 a linear expression of the eigenvalues of the Hessian instead of quotients of elementary symmetric functions.
\end{abstract}

\maketitle

\section{Introduction}

A constant rank theorem  asserts that a convex solution $u$ of an
elliptic partial differential equation, satisfying appropriate
conditions,  must have constant rank.  In the 1980's
Caffarelli-Friedman \cite{CF} proved such a result for semi-linear
elliptic equations, and a similar result was discovered around the same
time by Yau (see \cite{SWYY}).  These results were extended to more
general elliptic and parabolic PDEs by Korevaar-Lewis \cite{KL},
Caffarelli-Guan-Ma \cite{CGM} and Bian-Guan \cite{BG, BG2}.  Moreover, the
constant rank theorem (also known as the ``microscopic convexity
principle'') has been shown to hold for a number of geometric
differential equations involving the second fundamental form of
hypersurfaces \cite{GM, GLM, GMZ, CGM, BG}.  In addition, these ideas have  been
investigated in the complex
setting \cite{L, GLZ, HMW,GP}, where there are applications to K\"ahler
geometry. Constant rank theorems are also closely related to the question of convexity of
 solutions of non-linear PDE on convex domains (the ``macroscopic convexity principle'')  \cite{BL, CSp,K, Ke, Ka, ALL,MX,   W}. 

A common approach for establishing a constant rank theorem is to consider expressions involving the elementary symmetric polynomials $\sigma_{\ell}$ of the eigenvalues $\lambda_1 \le \cdots \le \lambda_n$ of the Hessian $D^2u$.  Indeed Bian-Guan \cite{BG} proved a rather general constant rank theorem for nonlinear elliptic equations $$F(D^2u, Du, u, x)=0$$ subject to a local convexity condition for $F$ (see (\ref{convexcondition}) below).  Their proof relies on a sophisticated computation using the quantity $\sigma_{\ell+1} + \frac{\sigma_{\ell+2}}{\sigma_{\ell+1}}$.  In this paper, we 
take a different approach by computing directly with the eigenvalues of $D^2u$ (compare with the works \cite{WY, Sz, STW}, for example).  We consider the simple linear expression
\begin{equation} \label{expression}
\lambda_{\ell} + 2 \lambda_{\ell-1} + \cdots + \ell \lambda_1,
\end{equation}
of the smallest $\ell$ eigenvalues of $D^2u$ (more precisely, we perturb $u$ slightly first).  While this expression is not smooth in general, it has the crucial property that it is \emph{semi-concave}, as long as $u$ is sufficiently regular. 

We now describe our result more precisely.
Let $\Omega$ be a domain in $\mathbb{R}^n$.   Write $\Sym(n)$ for the space of real symmetric $n \times n$ matrices, and $\Sym^+(n)$ for the subset that are strictly positive definite.
We consider the real-valued function $$F=F(A, p, u, x) \in C^{2}( \Sym(n) \times \mathbb{R}^n \times \mathbb{R} \times \Omega)$$ which satisfies the condition that for each $p \in \mathbb{R}^n$,
\begin{equation} \label{convexcondition}
(A, u,x) \in \Sym^+(n) \times \mathbb{R} \times \Omega \mapsto F(A^{-1}, p, u, x) \quad \textrm{is locally convex}.
\end{equation}
Now let $u \in C^{3}(\Omega)$ be a convex solution of
\begin{equation} \label{equation}
F(D^2 u, Du, u, x) =0,
\end{equation}
subject to the ellipticity condition
\begin{equation} \label{ellipticity}
F^{ij}(D^2 u, Du, u, x) >0 \quad \textrm{on } \Omega,
\end{equation}
where we write $F^{ij}$ for the derivative of $F$ with respect to the $(i,j)$th entry $A_{ij}$ of $A$.  Our main result is a new proof of the following theorem of Bian-Guan \cite{BG}.

\begin{theorem} \label{theorem1}
With $u$ and $F$ as above, satisfying (\ref{convexcondition}),  (\ref{equation}) and (\ref{ellipticity}), the Hessian $D^2u$ has constant rank in $\Omega$.
\end{theorem}

We give the proof in Section \ref{sectionpfthm1} below.  The heart of the proof is to establish a differential inequality (see (\ref{differentialinequality}) below) for our expression (\ref{expression}) and then apply a standard Harnack inequality.  Once we have perturbed $u$ so that the eigenvalues of its Hessian are distinct,  this is a straightforward computation (simpler than the analogous calculation in \cite{BG}).  The rest of the proof is concerned with making this formal argument rigorous.

If we replace the condition (\ref{convexcondition}) with a stronger
``strict convexity'' condition, we can prove the following additional
consequence (cf. \cite{BG}). Let us write $n-k$ for the rank of
$D^2u$, which is now constant.
Then there exist $k$ fixed directions $X_1, \ldots, X_{k}$ such
that $(D^2u(x))(X_j)=0$ for all $1 \le j \le k$ and all $x \in \Omega$. We
show this in Section \ref{sectionstrict}.   Examples (including one of
Korevaar-Lewis \cite{KL}) show that a stronger condition than
(\ref{convexcondition}) is indeed necessary for this conclusion. 

We expect that our techniques can also be used to give new proofs of constant rank theorems for 
 parabolic and geometric equations. 
 
\section{Proof of Theorem \ref{theorem1}} \label{sectionpfthm1}
As in \cite{BG}, 
note that the convexity condition (\ref{convexcondition}) can be written as follows:  for every symmetric matrix $(X_{ab}) \in \Sym(n)$, vector $(Z_a) \in \mathbb{R}^n$ and $Y \in \mathbb{R}$, we have
\begin{equation} \label{keyco}
\begin{split}
0 \le {} &  F^{ab,rs} X_{ab} X_{rs} + 2   F^{ar} A^{bs} X_{ab} X_{rs} +  F^{x_a, x_b} Z_a Z_b \\
& {} - 2 F^{ab,u} X_{ab} Y - 2  F^{ab, x_r} X_{ab} Z_r + 2 F^{u,x_a} Y Z_a + F^{u,u} Y^2, 
\end{split}
\end{equation}
where we are evaluating the derivatives $F$ at $(A, p, u, x)$ for a positive definite matrix $A$.  Here, we are using the usual notation for derivatives of $F$ (see \cite{BG}), we write $A^{ij}$ for the  $(i,j)$th entry of $A^{-1}$, and we use the standard convention of  summing repeated indices from $1$ to $n$.

To prove Theorem \ref{theorem1}, it is sufficient to prove the following. Suppose that at $x_0 \in \Omega$, the Hessian $D^2u$ has at least $k$ zero eigenvalues.  Then  there exists $r_0>0$ such that the Hessian $D^2u$ has at least $k$ zero eigenvalues 
 on the ball $B_{r_0}(x_0) \subset \Omega$ of radius $r_0$ centered at $x_0$.

We fix then this point $x_0 \in \Omega$, and write $B=B_{r_0}(x_0)$ for a sufficiently small $r_0>0$ (which we may shrink later) so that $B \subset \Omega$.

As pointed out in \cite{BG}, it follows from our assumptions on $u$ and $F$ and the standard elliptic regularity theory 
that $u$ is in $W^{4,p}(B)$ for all $p$.  We fix, once and for all, $p$ strictly larger than $n$.  Let $\ve>0$.  Since the polynomials are dense in $W^{4,p}(B)$, we can find a polynomial $P$ such that
\begin{equation}
\| P - u \|_{W^{4,p}(B)} \le \ve.
\end{equation}
By the Sobolev embedding theorem, 
\begin{equation} \label{C3alpha}
\| P - u\|_{C^{3, \alpha}(B)} \le C \ve,
\end{equation}
for uniform constants $C>0$ and $\alpha \in (0,1)$.

Now a generic polynomial $P$ will have the property that $D^2P$ has distinct eigenvalues away from a proper real analytic subset.  Note also that by the convexity of $u$, the Hessian $D^2u$ is nonnegative definite.  Hence, by making a small perturbation to $P$, we may assume without loss of generality that the eigenvalues $\Lambda_1 \le \cdots \le \Lambda_n$ of $D^2P$ are \emph{positive and distinct} away from a proper real analytic subset $V \subset B$.  At $x\in B \setminus V$ we have
$$0 < \Lambda_1 < \cdots < \Lambda_n.$$


We consider, for $\ell =1, \ldots, k$, the positive quantity
$$Q^{(\ell)} = \Lambda_{\ell} + 2 \Lambda_{\ell-1} + \cdots + \ell \Lambda_1 = \sum_{j=1}^{\ell} (\ell+1-j) \Lambda_j.$$
We will prove that on $B \setminus V$ (after possibly shrinking the radius of $B$), and for each $\ell =1, \ldots, k$,
\begin{equation} \label{ell}
Q^{(\ell)} + | D Q^{(\ell)}| \le c_{\ve},
\end{equation}
where we write $c_{\ve}$ to mean a constant satisfying $c_{\ve} \rightarrow 0$ as $\ve \rightarrow 0$.  Once (\ref{ell}) holds for $\ell=k$ we are done since then $Q^{(k)} \rightarrow 0$  on $B$ as $\ve \rightarrow 0$, which implies that the first $k$ eigenvalues of $D^2u$ must vanish everywhere on $B$.

We prove (\ref{ell}) by a finite induction.  Assume it holds for $1, 2, \ldots, \ell-1$ (if $\ell=1$, we do not assume anything).  Write $Q=Q^{(\ell)}$.  We first show that $Q$ satisfies the following differential inequality on $B \setminus V$ (on which $Q$ is smooth),
\begin{equation} \label{differentialinequality}
F^{ab}|_P Q_{ab} \le C | DQ| + CQ + f_{\ve},
\end{equation}
where $f_{\ve}$ has the property that $\| f_{\ve} \|_{L^n(B)} \rightarrow 0$ as $\ve \rightarrow 0$, and for a uniform $C$.   In what follows, we will denote by
 $c_{\ve}, f_{\ve}, C$ any quantities with the same properties as described here, where the uniformity will be clear from the context.
 
We compute at a fixed point $x \in B \setminus V$, and we assume that $D^2P$ is diagonal at this point with the eigenvalue $\Lambda_j$ given by $P_{jj}$.  Then the first derivative of $\Lambda_j$ is given at $x$ by
\begin{equation} \label{LFD}
(\Lambda_j)_a = P_{jja}.
\end{equation}
The inductive hypothesis tells us that $|D\Lambda_j| \le c_{\ve}$ for $j=1, \ldots, \ell-1$, and hence at $x$,
\begin{equation} \label{indhyp}
|P_{jji}| \le c_{\ve}, \quad \textrm{for all } j=1, \ldots, \ell-1, \ i=1, \ldots, n.
\end{equation}

We recall (see \cite{Sp} for example) that the second derivative of the eigenvalue $\Lambda_j$ of $D^2P$ at $x$ is given by
$$(\Lambda_j)_{ab} = P_{jjab} + 2 \sum_{m \neq j} \frac{P_{maj} P_{mbj}}{\Lambda_j-\Lambda_m}.$$
Hence
\begin{equation} \label{Qaa0}
\begin{split}
Q_{ab} = {} & \sum_{j=1}^{\ell} (\ell+1-j) P_{jjab} + 2\sum_{j=1}^{\ell} \sum_{m \neq j} (\ell+1-j) \frac{P_{maj}P_{mbj}}{\Lambda_j - \Lambda_m}\\
 = {} & \sum_{j=1}^{\ell} (\ell+1-j) P_{jjab}+  2 \sum_{1 \le j < m \le \ell} (m-j) \frac{P_{maj}P_{mbj}}{\Lambda_j-\Lambda_m} \\
 {} &+ 2\sum_{j=1}^{\ell} \sum_{m>\ell} (\ell+1-j) \frac{P_{maj}P_{mbj}}{\Lambda_j - \Lambda_m},
\end{split}
\end{equation}
where the second line is obtained after cancelling the positive terms with $\Lambda_j-\Lambda_m$ for $j>m$ with the corresponding negative terms. 

We now differentiate the equation (\ref{equation}) twice in the $j$ direction to obtain
\begin{equation} \label{deru}
\begin{split}
0= {} & F^{ab}|_u u_{abjj} + F^{p_a}|_u u_{ajj} + F^u|_u u_{jj} \\
{} & + F^{ab, rs}|_u u_{abj} u_{rsj} + F^{p_a,p_b}|_u u_{aj} u_{bj} + F^{u,u}|_u u_j^2 + F^{x_j, x_j}|_u \\
{} & +
2F^{ab, p_r}|_u u_{abj} u_{rj}+ 2F^{ab, u}|_u u_{abj} u_j + 2F^{ab, x_j}|_u u_{abj} \\
{} & + 2 F^{p_a, u}|_u u_{aj} u_j + 2 F^{p_a, x_j}|_u u_{aj} + 2F^{u, x_j}|_u u_j.
\end{split}
\end{equation}
Replace $u$ by the polynomial $P$, at the expense of an error term $f_{\ve}$, to get
\begin{equation}
\begin{split}
f_{\ve} = {} & F^{ab} P_{abjj} + F^{p_a} P_{ajj} + F^u P_{jj} \\
{} & + F^{ab, rs} P_{abj} P_{rsj} + F^{p_a,p_b} P_{aj} P_{bj} + F^{u,u} P_j^2 + F^{x_j, x_j} \\
{} & +
2F^{ab, p_r} P_{abj} P_{rj}+ 2F^{ab, u} P_{abj} P_j + 2F^{ab, x_j} P_{abj} \\
{} & + 2 F^{p_a, u} P_{aj} P_j + 2 F^{p_a, x_j} P_{aj} + 2F^{u, x_j} P_j,
\end{split}
\end{equation}
where, here and for the rest of  this section, $F^{ab}, F^{p_a}$ etc are all evaluated at $P$, and
$\Vert f_\ve\Vert_{L^n(B)} \to 0$ as $\ve \to 0$.  Combining this with (\ref{Qaa0}), we obtain
\begin{equation} \label{FQ}
\begin{split}
F^{ab}Q_{ab} \le  {} &   
 2 \sum_{1 \le j < m \le \ell} (m-j) F^{ab}
 \frac{P_{maj}P_{mbj}}{\Lambda_j-\Lambda_m}  \\
&+ 2\sum_{j=1}^{\ell} \sum_{m>\ell} (\ell+1-j) F^{ab}\frac{P_{maj}P_{mbj}}{\Lambda_j - \Lambda_m} \\ {} & + C\bigg( \sum_{i=1}^n \sum_{a,b=1}^{\ell} |P_{abi}| \bigg) + CQ+ f_{\ve} + (*),
\end{split}
\end{equation}
where
\[
\begin{split}
(*)={} & - \sum_{j=1}^{\ell} (\ell+1-j) \bigg\{ \sum_{a,b,r,s=\ell +1}^n F^{ab,rs} P_{abj} P_{rsj} + F^{u,u}P_j^2 + F^{x_j,x_j} \\
{} & + 2\sum_{a,b=\ell+1}^n F^{ab,u} P_{abj}P_j + 2 \sum_{a,b=\ell+1}^n F^{ab, x_j} P_{abj} + 2F^{u, x_j} P_j \bigg\}.
\end{split}
\]
Note that we have separated out all terms which involve $P_{abj}$ with
at least two indices between $1$ and $\ell$, as well as all terms
involving the eigenvalues $P_{aa}$ for $a \leq \ell$. 

For each fixed $j$, we now use (\ref{keyco}), with
\begin{equation} \label{choice}
X_{pq} = \left\{ \begin{array}{ll} - P_{pqj} & \text{ if }p,q > \ell \\
0  & \text{ otherwise}, \end{array} \right.
\quad  
Z_i = \left\{ \begin{array}{ll} 1 & \text{ if } i=j \\ 0 & \text{ otherwise,} \end{array} \right.
\quad Y = P_j.
\end{equation}
This implies
\[
\begin{split}
0 \le {} & \sum_{a,b,r,s=\ell+1}^n F^{ab,rs} P_{abj} P_{rsj} + 2 \sum_{a,b,m=\ell+1}^n F^{ab} \frac{P_{maj}P_{mbj}}{\Lambda_m} + F^{x_j,x_j} \\
{} & + 2\sum_{a,b=\ell+1}^n F^{ab,u} P_{abj} P_j + 2 \sum_{a,b=\ell+1}^n F^{ab,x_j} P_{abj} + 2F^{u,x_j} P_j + F^{u,u} P_j^2,
\end{split}
\]
and hence, using that $0 < \Lambda_m-\Lambda_j < \Lambda_m$ whenever
$j < m$,
\begin{equation} \label{num1}
\begin{split}
(*) \le {} & 2 \sum_{j=1}^{\ell} \sum_{a,b,m=\ell+1}^n (\ell+1-j) F^{ab} \frac{P_{maj}P_{mbj}}{\Lambda_m} \\
\le {} & 2\sum_{j=1}^{\ell} \sum_{m > \ell} (\ell+1-j) F^{ab}
\frac{P_{maj}P_{mbj}}{\Lambda_m - \Lambda_j}.
\end{split}
\end{equation}
On the other hand, for a uniform $c > 0$ we have
\[ 2 \sum_{1 \leq j < m \leq \ell} (m-j) F^{ab}
\frac{P_{maj}P_{mbj}}{\Lambda_m - \Lambda_j} \geq cQ^{-1} \sum_{i=1}^n \sum_{1\leq a < b
  \leq \ell} P_{abi}^2, \]
using the ellipticity assumption, and that $Q \geq \Lambda_m -
\Lambda_j$ whenever $ j < m \leq l$. Using this, together with the
inequality
\[ C|P_{abi}| \leq cQ^{-1} P_{abi}^2 + C'Q, \]
it follows that
\begin{equation} \label{num2}
\begin{split} C\sum_{i=1}^n \sum_{a,b=1}^{\ell} |P_{abi}| &\le 
  C\sum_{i=1}^n \sum_{a=1}^{\ell} |P_{aai}| \\
&\quad + 2 \sum_{1 \le j < m \le \ell} (m-j) F^{ab}
\frac{P_{maj}P_{mbj}}{\Lambda_m-\Lambda_j} +C'Q,
\end{split}
\end{equation}
for suitable $C'$. 
Then
combining (\ref{FQ}) with (\ref{num1}), (\ref{num2}), and making use
of the inductive hypothesis (\ref{indhyp}), we obtain 
\[
\begin{split}
F^{ab}Q_{ab} \le {} & C\sum_{i=1}^n \sum_{a=1}^{\ell} |P_{aai}| + CQ + f_{\ve} \\
\le {} & C \sum_{i=1}^n |Q_i| + CQ + f_{\ve}.
\end{split}
\]
Namely, the differential inequality (\ref{differentialinequality}) holds on $B \setminus V$.

We now wish to apply the Harnack inequality to $Q$.  
Note however that $Q$ may not be smooth on $V$.  On the other hand, $Q$ is a semi-concave function on the whole of $B$ (see for example \cite[p.40]{CS}) which means that $Q = U+W$ where $U$ is concave on $B$ and $W \in C^{1,1}(B)$.  We will use  following version of the Harnack inequality (cf. \cite[Lemma 2.1.(III)]{Tr}, where the assumptions are slightly different).

\begin{lemma}  Consider the operator $L$ given by $Lv = a^{ij} D_{ij}v + b^i D_iv  + cv$ with bounded coefficients, with $a^{ij}$ satisfying $\lambda |\xi|^2 \le a^{ij} \xi_i \xi_j \le \Lambda |\xi|^2$ for all $\xi \in \mathbb{R}^n$, with $\lambda, \Lambda>0$.
Let $v$ be a semi-concave nonnegative function on the ball $B\subset\mathbb{R}^n$ which is smooth on $B \setminus V$, where $V$ is a proper real analytic variety.  Suppose that  $Lv \le f$ in $B$ where $f \in L^n(B)$.  Then on the half size ball $B'$,
\begin{equation}
\left( \frac{1}{|B'|} \int_{B'} v^q \right)^{1/q} \le C \left( \inf_{B'} v + \| f\|_{L^n(B)} \right),
\end{equation}
for positive constants $C$ and $q$ depending only on $n$, $\lambda, \Lambda$, bounds for $b^i$ and $c$ and the radius of the ball $B$.
\end{lemma}
\begin{proof}  We give the proof for the reader's convenience.  If $v \in W^{2,n}(B)$, then this result is standard (see for example \cite[Theorem 9.22]{GT}).  We prove the result we need for $v$ semi-concave using a mollification argument.
Let $v_{\ve}$ denote a standard mollification of
$v$, with a mollifier whose support has radius $\ve>0$. Let $\delta
> 0$. Outside the
$\delta$-neighborhood of $V$ we have a bound $\sum_{|\gamma| \le 3} |D^{\gamma}v| \le
C_\delta$. For sufficiently small $\ve$ we will then have 
\[ \sum_{|\gamma| \le 2} |D^{\gamma}(v_\varepsilon - v)| \le \varepsilon C_\delta, \]
away from the $2\delta$-neighborhood of $V$. Applying the differential
inequality $Lv \le f$ outside the $2\delta$-neighborhood of $V$ we will have
\[  Lv_{\ve} = Lv+ L(v_{\ve}-v)  \leq f + C \ve C_\delta  \]
where $C$ depends on the bounds for the coefficients of $L$.

On the other hand, since $v$ is semi-concave we have 
a fixed upper bound for $D^2v_{\ve}$ everywhere.   In particular,  near $V$ we have
$L(v_\ve) \leq C$. Since the $2\delta$-neighborhood of $V$ has
measure at most $C\delta$ for some fixed $C$, applying the standard Harnack
inequality to the smooth function $v_{\ve}$, we obtain
\[ \left( \frac{1}{|B'|} \int_{B'} v_\ve^q\right)^{1/q} \leq C( \inf_{B'} v_\ve + \| f\|_{L^n(B)} + 
\ve C_{\delta} + \delta^{1/n}). \]
 We can then first choose $\delta$ very small,
and then let $\ve \rightarrow 0$, to obtain the required inequality for $v$. 
\end{proof}

Applying this lemma to $Q$, we obtain
$$\left( \frac{1}{|B'|}\int_{B'} Q^q \right)^{1/q}\le C (\inf_{B'} Q + \| f_{\ve}\|_{L^n(B)}) \le c_{\ve},$$
since $\inf_{B'} Q$ tends to zero as $\ve \rightarrow 0$. From (\ref{LFD}) we have 
 a uniform Lipschitz bound on $Q$, and this implies that $|Q| \leq
c_\epsilon$ on $B'$. 

The required bound on $|DQ|$ in
(\ref{ell}) then follows from the next lemma, which uses again the  semi-concavity of $Q$.  Recall from (\ref{C3alpha}) that we have a uniform bound on the $C^{3,\alpha}$ norm of $P$.

\begin{lemma}
  There is a constant $C$ depending on the $C^{3,\alpha}$ norm of $P$
  with the following property.  For $x\in \frac{1}{2}B' \setminus V$ and
  sufficiently small $\varepsilon$,  we have
  \begin{equation} \label{DQbound}
   |DQ(x)|^{1+1/\alpha} \leq C c_\varepsilon, 
   \end{equation}
  where $c_\varepsilon = \sup_{B'} Q$.   Here $\alpha \in (0,1)$ is the constant from (\ref{C3alpha}).
\end{lemma}
\begin{proof} Let $x \in \frac{1}{2}B' \setminus V$ and let $\xi$ be a unit vector for which $D_{\xi} Q(x)<0$.  We already know that $D_{\xi}Q(x)$ is bounded, since $Q$ is uniformly Lipschitz, and our goal is to obtain the stronger bound (\ref{DQbound}).
For $r>0$ sufficiently small (to be determined later), define $y = x+ r\xi \in B'$ and $$x_t = (1-t)x+ty = x+  rt \xi, \quad \textrm{for } t\in [0,1].$$

Now define the function $h : \mathrm{Sym}(n) \to \mathbf{R}$ by
  \[ h(A) = \ell \lambda_1(A) + \ldots + \lambda_\ell(A), \]
  so that $Q(z) = h(D^2P(z))$ for any $z$. Note that $h$ is a concave Lipschitz
  function. Using the concavity of $h$ we have
  \begin{equation}\label{eq:h}
    \begin{aligned}
      (1-t)Q(x) + tQ(y) &= (1-t) h(D^2P(x)) + t h(D^2P(y)) \\
      &\leq h\left( (1-t)D^2P(x) + tD^2P(y)\right).
    \end{aligned}
  \end{equation}
Next,
  \[ \begin{split}
  \lefteqn{  \left| (1-t)D^2P(x) + tD^2P(y) - D^2P(x_t) \right| } \\
 = {} &  \left| (1-t)\left( D^2P(x) - D^2P(x_t)\right)  + t  \left( D^2P(y) - D^2P(x_t) \right) \right| \\
 \le {} & Ct(1-t)  |y-x|^{1+\alpha},   
    \end{split}
    \]
by applying the Mean Value Theorem to $D^2P(x)-D^2P(x_t)$ and $D^2P(y)-D^2P(x_t)$ and then using the fact that $D^3P$ has bounded
 $\alpha$-H\"older norm.
  Using this in \eqref{eq:h}, and writing $h(D^2P(x_t)) = Q(x_t)$,  we have
  \[ (1-t) Q(x) + t Q(y) \leq Q(x_t) + C t(1-t) |y-x|^{1+\alpha}, \]
  where we also used the Lipschitz property of $h$. 
  This implies
  \begin{equation} \label{eq:DQ} \frac{Q(y) - Q(x)}{|y-x|} \leq \frac{ Q(x_t) - Q(x) }{t |y-x|} +
    C_1|y-x|^\alpha.
  \end{equation}  
  Letting $t \rightarrow 0$ and recalling that $c_{\ve} = \sup_{B'}Q$, we have
  $$- \frac{c_{\ve}}{r} \le D_{\xi}Q(x) + C_1 r^{\alpha},$$
  for a uniform $C_1$.
Choose $r>0$ so that $C_1r^{\alpha} = -D_{\xi}Q(x)/2$ (increasing $C_1$ if necessary to ensure that $B_r(x) \subset B'$).  We obtain after rearranging,
$$|D_{\xi}Q (x)|^{1+1/\alpha}  \le 2c_{\ve}(2C_1)^{1/\alpha},$$
as required.
\end{proof}

This completes the proof of (\ref{ell}) and hence Theorem \ref{theorem1}.

\section{Strict convexity} \label{sectionstrict}

We consider now the case when we replace the condition (\ref{convexcondition}) by a strict convexity type condition.  Recall from Section \ref{sectionpfthm1} that (\ref{convexcondition}) is equivalent to (\ref{keyco}).  We now consider the condition: there exists  a continuous function $\eta>0$ on $\Omega$ such that
 for every symmetric matrix $(X_{ab}) \in \Sym(n)$, vector $(Z_a) \in \mathbb{R}^n$ and $Y \in \mathbb{R}$, we have
 \begin{equation} \label{strict}
\begin{split}
\eta |X|^2 \le {} &  F^{ab,rs} X_{ab} X_{rs} + 2   F^{ar} A^{bs} X_{ab} X_{rs} +  F^{x_a, x_b} Z_a Z_b \\
& {} - 2 F^{ab,u} X_{ab} Y - 2  F^{ab, x_r} X_{ab} Z_r + 2 F^{u,x_a} Y Z_a + F^{u,u} Y^2, 
\end{split}
\end{equation}
where we are evaluating the derivatives $F$ at $(A, p, u, x)$ for a positive definite matrix $A$. Namely we replace the $0$ on the left hand side of (\ref{keyco}) by $\eta |X|^2$.

Then we have:

\pagebreak[3]
\begin{theorem} \label{theorem2}
Let $u$ and $F$ be as in Theorem \ref{theorem1}, except that (\ref{convexcondition}) is replaced by the stronger condition (\ref{strict}).  Let the rank of $D^2u$ be $n-k$.  Then  there exist  $k$ fixed directions $X_1, \ldots, X_{k}$ such that $(D^2u(x))(X_j)=0$ for all $1 \le j \le k$ and all $x \in \Omega$. 
\end{theorem}
\begin{proof}  
Write $0 \le \lambda_1 \le \cdots \le \lambda_n$ for the eigenvalues of $D^2u$.  
Our assumption is that $\lambda_1= \cdots= \lambda_k =0$ and $\lambda_{k+1} > 0$ on $\Omega$.  Fix $x_0 \in \Omega$.  It suffices to show that there exists a neighborhood $U$ of $x_0$ and fixed directions 
 $X_1, \ldots, X_k$ such that $(D^2u(x))(X_j)=0$ for $1 \le j \le k$ and all $x \in U$.
Moreover, by an argument of Korevaar-Lewis  \cite[p. 29-30]{KL} it is enough to show that, for $x$ near $x_0$,  we have $D_X D^2u(x) =0$ for all vectors $X$ in the null space of $D^2u(x)$.  
Fix $x$ and choose coordinates such that $D^2u(x)$ is diagonal with $u_{ii}= \lambda_i$.  We wish to show that $u_{pqi}(x)=0$ for $1\le i \le k$ and all $p,q$.

First, let $Y$ be a vector in the null space of $D^2u(x)$ and extend to a constant vector field in a neighborhood of $x$.
Then 
$D^2u(Y,Y)$ is a nonnegative function which vanishes at $x$.  It follows that its first derivative vanishes at $x$, and so $u_{pqi}Y^pY^q=0$ at $x$. Hence 
\begin{equation} \label{upqi}
u_{pqi}(x)=0 \quad \textrm{for } 1 \le p, q\le k \textrm{ and all $i$}.
\end{equation}

We now consider the quantity $R = \lambda_1 + \ldots + \lambda_k$, the
sum of the $k$ smallest eigenvalues.  Since $R=0$ and $\lambda_{k+1}>0$ on $\Omega$, 
it follows that
we can differentiate $R$ (more precisely, the ``sum of first $k$
eigenvalues'' function on the space of symmetric matrices is smooth
in a neighborhood of the set of values of $D^2u$), and we obtain
\[ 0 = R_{ab} = \sum_{j=1}^k u_{jjab} + 2\sum_{j=1}^k \sum_{m \ge k+1}
\frac{u_{maj}u_{mbj}}{ - \lambda_m}. \]
Taking the trace of this equation with respect to $F^{ab}$ (evaluating at $u$) and using (\ref{deru}) and (\ref{upqi}) gives
\[ 
\begin{split}
0 = {} &  \sum_{j=1}^k \bigg\{ \sum_{a,b,r,s \ge k+1} F^{ab,rs} u_{abj} u_{rsj} + F^{u,u} u_j^2+ F^{x_j,x_j} 
+ 2\sum_{a,b \ge k+1} F^{ab,u} u_{abj} u_j \\
{} & + 2\sum_{a,b \ge k+1}F^{ab,x_j}u_{abj} +2F^{u,x_j} u_j  + 2\sum_{a,b,m \ge k+1}F^{ab} \frac{u_{maj} u_{mbj}}{\lambda_m} \bigg\}.
\end{split}
\]
We now apply the strict convexity assumption at $u$ using the choices (\ref{choice}) with $P$ replaced by $u$ and $\ell$ replaced by $k$.  Note that since $D^2u$ is not strictly positive definite, in (\ref{strict}) we take positive definite matrices $A$ with  $A \rightarrow D^2u$.  This implies
\[ 0 \ge \eta \sum_{j=1}^k \sum_{p, q \ge k+1}
|u_{pqj}|^2, \]
and hence $u_{pqi}=0$ for $i \leq k$ and $p,q > k$. Together with (\ref{upqi}), this completes the proof of the theorem.
\end{proof}

To see that the assumptions of Theorem \ref{theorem1} are not sufficient for the stronger conclusion of Theorem \ref{theorem2}, one can consider the example of Korevaar-Lewis \cite[p. 31]{KL} which corresponds to $F(A,u) = u-1/\textrm{tr}(A)$.  In our setting, a similar but slightly simpler example is given by
\begin{equation} \label{example}
F(A,x) = r-\frac{n-1}{\textrm{tr}(A)},
\end{equation}
where $r = \sqrt{x_1^2 + \cdots +x_n^2}$, with $\Omega$ a small ball which does not intersect the origin in $\mathbb{R}^n$.   
A solution of the equation $F(D^2u,x)=0$ is given by $u= r$, which is linear along different lines at different points.  One can check that $F$ given by (\ref{example}) satisfies (\ref{convexcondition}), but not the stronger condition (\ref{strict}).


\begin{thebibliography}{99}
\bibitem{ALL} Alvarez, O., Lasry, J.-M., Lions, P.-L. {\em Convex
    viscosity solutions and state constraints}, J. Math. Pures
  Appl. (9) 76 (1997), no. 3, 265--288
\bibitem{BG} Bian, B., Guan, P. {\em 
A microscopic convexity principle for nonlinear partial differential equations}, Invent. Math. 177 (2009), 307--335
\bibitem{BG2} Bian, B., Guan, P. {\em A structural condition for microscopic convexity principle}, Discrete Contin. Dyn. Syst. 28 (2010), no. 2, 789--807
\bibitem{BL}  Brascamp, H.J., Lieb, E.H. {\em On extensions of the Brunn-Minkowski and Pr\'ekopa-Leindler theorems, including inequalities for log concave functions, and with an application to the diffusion equation}, J. Functional Analysis 22 (1976), no. 4, 366--389
\bibitem{CF} Caffarelli, L., Friedman, A. {\em Convexity of solutions of some semilinear elliptic equations}, Duke Math. J. 52 (1985), 431--455
\bibitem{CGM} Caffarelli, L., Guan, P., Ma, X. {\em A constant rank
    theorem for solutions of fully nonlinear elliptic equations},
  Commun. Pure Appl. Math. 60 (2007), 1769--1791
\bibitem{CSp} Caffarelli, L., Spruck, J. {\em Convexity properties of
    solutions to some classical variational problems}, Comm. Partial
  Differential Equations 7 (1982), 1337--1379
\bibitem{CS} Cannarsa, P., Sinestrari, C. {\em Semiconcave functions, Hamilton-Jacobi equations, and optimal control}, Progress in Nonlinear Differential Equations and their Applications 58, Birkh\"auser Boston, Inc., Boston, MA, 2004
\bibitem{GT} Gilbarg, D., Trudinger, N.S. {\em Elliptic partial differential equations of second order}, Reprint of the 1998 edition. Classics in Mathematics. Springer-Verlag, Berlin, 2001
\bibitem{GLZ} Guan, P., Li, Q., Zhang, X. {\em A uniqueness theorem in K\"ahler geometry}, Math. Ann. 345 (2009), no. 2, 377--393
\bibitem{GLM} Guan, P., Lin, C.S., Ma, X.N. {\em The Christoffel-Minkowski problem II: Weingarten curvature equations}, Chin. Ann. Math., Ser. B 27 (2006), 595--614
\bibitem{GM} Guan, P., Ma, X.N. {\em The Christoffel-Minkowski problem I: Convexity of solutions of a Hessian equations}, Invent. Math. 151 (2003), 553--577
\bibitem{GMZ} Guan, P., Ma, X.N., Zhou, F. {\em The Christoffel-Minkowski problem III: existence and convexity of admissible solutions}, Commun. Pure Appl. Math. 59 (2006), 1352--1376
\bibitem{GP} Guan, P., Phong, D.H. {\em A maximum rank problem for degenerate elliptic fully nonlinear equations}, Math. Ann. 354 (2012), no. 1, 147--169
\bibitem{HMW} Han, F., Ma, X.-N., Wu, D. {\em A constant rank theorem for Hermitian $k$-convex solutions of complex Laplace equations}, Methods Appl. Anal. 16 (2009), no. 2, 263--289
\bibitem{Ka} Kawohl, B. {\em 
A remark on N. Korevaar's concavity maximum principle and on the asymptotic uniqueness of solutions to the plasma problem}, 
Math. Methods Appl. Sci. 8 (1986), no. 1, 93--101
\bibitem{Ke} Kennington, A.U. {\em Power concavity and boundary value problems},  Indiana Univ. Math. J. 34 (1985), no. 3, 687--704
\bibitem{K}  Korevaar, N.J. {\em Capillary surface convexity above convex domains}, Indiana Univ. Math. J. 32 (1983), 73--81
\bibitem{KL} Korevaar, N.J., Lewis, J.L. {\em Convex solutions of certain elliptic equations have constant rank Hessians}, 
Arch. Rational Mech. Anal. 97 (1987), no. 1, 19--32.
\bibitem{L} Li, Q. {\em Constant rank theorem in complex variables},
  Indiana Univ. Math. J. 58 (2009), no. 3, 1235--1256
\bibitem{MX} Ma, X.-N., Xu, L. {\em The convexity of solutions of a
    class of Hessian equation in bounded convex domain in
    $\mathbb{R}^3$}, J. Funct. Anal. 255 (2008), no. 7, 1713--1723 
\bibitem{SWYY} Singer, I., Wong, B., Yau, S.T., Yau, S.S.T. {\em An estimate of gap of the first two eigenvalues in the
Schrodinger operator}, Ann. Scuola Norm. Sup. Pisa Cl. Sci. (4) 12 (1985), 319--333
\bibitem{Sp} Spruck, J. {\em Geometric aspects of the theory of fully nonlinear elliptic equations, Global
theory of minimal surfaces}, Amer. Math. Soc., Providence, RI, 2005, 283--309
\bibitem{Sz} Sz\'ekelyhidi, G. {\em Fully non-linear elliptic equations on compact Hermitian manifolds}, preprint, arXiv:1501.02762
\bibitem{STW} Sz\'ekelyhidi, G., Tosatti, V., Weinkove, B. {\em 
Gauduchon metrics with prescribed volume form}, 
preprint, arXiv:1503.04491
\bibitem{Tr} Trudinger, N.S. {\em Comparison principles and pointwise
    estimates for viscosity solutions of nonlinear elliptic
    equations}, Revista Mat. Iber. 4 (1988), no. 3 \& 4, 453--468
\bibitem{W} Wang, X.-J. {\em Counterexample to the convexity of level
    sets of solutions to the mean curvature equation},
  J. Eur. Math. Soc.  16 (2014), no. 6, 1173--1182
 \bibitem{WY}  Warren, M., Yuan, Y. {\em Hessian estimates for the sigma-2 equation in dimension 3}, Comm. Pure Appl. Math. 62 (2009), no. 3, 305--321
 \end{thebibliography}
\end{document}